\documentclass[reqno,11pt]{amsart}
\usepackage[normalem]{ulem}
\PassOptionsToPackage{table}{xcolor}
\usepackage{stmaryrd}
\expandafter\def\csname opt@stmaryrd.sty\endcsname
{only,shortleftarrow,shortrightarrow}
\usepackage{extpfeil}
\usepackage{amsmath,amssymb,amsthm}
\usepackage{enumerate,mathrsfs, bbm}
\usepackage{url,titletoc,enumitem}
\usepackage{bbold}
\usepackage{makecell}
\usepackage{graphicx} 
\usepackage{aliascnt}
\usepackage{varioref}
\usepackage{hyperref}
\usepackage[capitalize,nameinlink,noabbrev]{cleveref}
\hypersetup{colorlinks=true,urlcolor=blue,citecolor=blue,linkcolor=blue}
\usepackage{courier}
\usepackage{tikz}
\usepackage{tikz-cd}
\usetikzlibrary{calc,matrix,arrows,decorations.markings}
\usepackage{array}
\usepackage{color}
\usepackage{enumerate}
\usepackage{nicefrac}
\usepackage{listings}
\usepackage{siunitx}
\usepackage{seqsplit}
\usepackage{longtable}
\usepackage{colonequals}
\usepackage{array, multirow}
\usepackage{booktabs}
\usepackage{comment}
\newenvironment{salign}{
    \begin{equation}
    \begin{aligned}
}{
    \end{aligned}
    \end{equation}
    \ignorespacesafterend
}

\newcommand{\defeq}{\stackrel{\textnormal{def}}{=}}

\DeclareMathOperator{\Aut}{Aut}

\DeclareMathOperator{\disc}{disc}

\DeclareMathOperator{\Gal}{Gal}

\DeclareMathOperator{\Ht}{Ht}
\DeclareMathOperator{\ord}{ord}
\DeclareMathOperator{\Beta}{\mathbf{B}}

\newtheorem{theorem}{Theorem}[section]

\newtheorem{lemma}[theorem]{Lemma}

\newtheorem{prop}[theorem]{Proposition}

\newtheorem{remark}[theorem]{Remark}

\newtheorem{conj}[theorem]{Conjecture}

\newtheorem*{lemma*}{Lemma}

\numberwithin{equation}{section}

\def\bbR{ {\mathbb R}}

\def\bbF{ {\mathbb F}}
\def\bbQ{ {\mathbb Q}}

\def\bbP{ {\mathbb P}}
\def\bbC{ {\mathbb C}}

\def\bb1{ {\mathbb 1}}

\def\cA{ {\mathcal A} }

\def\cC{ {\mathcal C} }

\def\cM{{\mathcal M}}
\def\cN{ {\mathcal N} }
\def\cO{ {\mathcal O} }
\def\cP{{\mathcal P}}

\def\cR{{\mathcal R}}

\def\cV{{\mathcal V}}

\def\cX{ {\mathcal X} }

\def\sI{ {\mathscr I}}
\def\sJ{ {\mathscr J}}

\def \w{ {\textbf w}}

\newcommand{\Z}{\mathbb{Z}}
\newcommand{\Q}{\mathbb{Q}}

\newcommand{\AL}{\textnormal{AL}}
\newcommand{\GL}{\textnormal{GL}}
\newcommand{\id}{\textnormal{id}}

\newcommand{\SL}{\textnormal{SL}}
\newcommand{\Val}{\textnormal{Val}}

\title{Counting points on some genus zero Shimura curves}

\author{\sc Tyler Genao}
\address{Tyler Genao \\
The Ohio State University\\
USA}
\urladdr{https://tylergenao.com/}
\email{genao.5@osu.edu}

\author{\sc Tristan Phillips}
\address{Tristan Phillips \\
Dartmouth College\\
USA}
\urladdr{https://sites.google.com/site/tristanphillipsmath/home}
\email{tristanphillips72@gmail.com}

\author{\sc Frederick Saia}
\address{Frederick Saia \\
University of Illinois Chicago\\
USA}
\urladdr{https://fsaia.github.io/site/}
\email{fsaia@uic.edu}

\author{\sc Tim Santens}
\address{Tim Santens \\
University of Cambridge\\
UK}
\urladdr{https://sites.google.com/view/timsantens/}
\email{ts996@cam.ac.uk}

\author{\sc John Yin}
\address{John Yin \\
The Ohio State University\\
USA}
\urladdr{https://sites.google.com/view/johng23}
\email{yin.1034@osu.edu}

\date{}

\begin{document}


\begin{abstract}
      We count certain abelian surfaces with potential quaternionic multiplication over a number field $K$ by counting points of bounded height on some genus zero Shimura curves.
\end{abstract}

\maketitle

\section{Introduction}

There has been significant recent interest in the problem of counting elliptic curves of bounded height over the rational numbers $\mathbb{Q}$, or more generally over a fixed number field $K$, which have a prescribed torsion group structure. Works following this trend \cite{HS17,PPV20,PS21,CKV22,BN22,MV23,BS24,Phi24+} take the perspective of counting points of bounded height on modular curves parameterizing elliptic curves with a specified structure. 

In this paper, we count points of bounded height on certain compact Shimura curves of genus zero. Let $\cX^D$ denote the Shimura curve defined over $K$ which parameterizes abelian surfaces with fixed quaternionic multiplication (QM) structure by the quaternion algebra over $\bbQ$ of discriminant $D$. This curve has genus zero for $D\in \{6,10,22\}$.
We count points on the corresponding Shimura curves $\cX^6$, $\cX^{10}$, and $\cX^{22}$, and also on their Atkin--Lehner quotients,
which parameterize abelian surfaces with potential quaternionic multiplication (PQM). This is a step towards more generally answering these types of counting problems for abelian varieties of higher dimension, accessed by restricting to a case where extra structure provides a one-dimensional moduli space like for modular curves. We count via the Igusa invariants of corresponding families of genus $2$ curves. For $D\in \{6,10\}$ these families of curves first appeared in work of Hashimoto--Murabayashi \cite{HM95} and have also been studied by Baba--Granath \cite{BG08}. For $D=22$, we use generalizations of these results due to Nualart Riera \cite{NR15} and Lin--Yang \cite{LY20}.
 
 For $D \in \{6,10,22\}$, let $\AL(\cX^D)$ denote the full group of Atkin--Lehner involutions on $\cX^D$, and let $\cA_2$ denote the moduli stack of principally polarized abelian surfaces. For each subgroup $W \leq \AL(\cX^D)$, \cite{HM95}, \cite{BG08}, and \cite{NR15}  describe families $C_{W}^D$ of genus $2$ curves corresponding to the image of $\cX^D/W$ in $\cA_2$ under the map forgetting the QM structure. We discuss these families further in \cref{shim_sec} and \cref{can_ring_sec}. For any number field $K$, the Igusa invariants give a map into a weighted projective stack
\[ 
C^D_{W}(K) \xrightarrow{\sI} \cP(1,2,3,5)(\overline{K}). 
\] 
From this map and a height function on the weighted projective stack we define an \textit{Igusa height} function $\overline{\Ht}$ on $\overline{K}$-isomorphism classes of genus two curves. We also define a height $\Ht$ on $K$-isomorphism classes of genus two curves in such a way that results on Malle's conjecture for counting field extensions can be used to count genus two curves in twist classes.\footnote{See \cref{sec:heights} for the precise definitions of our height functions.}
Let $B \in \mathbb{R}_{> 0}$, and let $\overline{\mathcal{N}}^{D}(W,K,B)$ (resp.\@ $\mathcal{N}^{D}(W,K,B)$) be the counting function for the number of $\overline{K}$-isomorphism classes (resp.\@ $K$-isomorphism classes) of curves $C$ in $C_{W}^{D}(K)$ with $K$ a field of definition and whose Igusa height $\overline{\Ht}(C)$ (resp.\@ whose height $\Ht(C)$) is bounded by $B$. With this setup, we state our main result.

\begin{theorem}\label{main_thm}
We have that
\begin{align*}
\overline{\mathcal{N}}^6(W,K,B)\asymp \mathcal{N}^6(W,K,B) &\ll
\begin{cases}
   B^2/\log(B) &\text{ if } W\in\{\langle w_3\rangle, \AL(\cX^6)\}, \\
    B^2/\log(B)^{1/2} &\text{ if } W\in\{\langle w_2\rangle, \langle w_6\rangle\},\\
    B^2 &\text{ if } W= \{\textnormal{id}\} \text{ and } (-6,2)_K=1, \\
    0 &\text{ if } W=\{\textnormal{id}\} \text{ and } (-6,2)_K=-1,
\end{cases}\\
\overline{\mathcal{N}}^{10}(W,K,B)
\ll&
\begin{cases}
   B/\log(B) &\text{ if } W\in\{\langle w_2\rangle, \AL(\cX^{10})\}, \\
    B/\log(B)^{1/2} &\text{ if } W\in\{\langle w_5\rangle, \langle w_{10}\rangle\},\\
    B &\text{ if } W=\{\textnormal{id}\} \text{ and } (-10,5)_K=1, \\
    0 &\text{ if } W=\{\textnormal{id}\} \text{ and } (-10,5)_K=-1,
\end{cases}\\
\mathcal{N}^{10}(W,K,B)\ll&
\begin{cases}
   B &\text{ if } W\in\{\langle w_2\rangle, \AL(\cX^{10})\}, \\
    B\log(B)^{1/2} &\text{ if } W\in\{\langle w_5\rangle, \langle w_{10}\rangle\},\\
    B\log(B) &\text{ if } W=\{\textnormal{id}\} \text{ and } (-10,5)_K=1, \\
    0 &\text{ if } W=\{\textnormal{id}\} \text{ and } (-10,5)_K=-1,
\end{cases}\\
\overline{\mathcal{N}}^{22}(W,K,B)\ll&
\begin{cases}
   B^{2/5}/\log(B)^{3/2} &\text{ if } W\in\{\langle w_2\rangle,  \AL(\cX^{22})\} \\
   &\text{ or if } W=\langle w_{11}\rangle \text{ and } \Q(\sqrt{-11}) \subseteq K,\\
    B^{2/5}/\log(B) &\text{ if } W = \langle w_{22}\rangle \\
    &\text{ or if } W=\{\id\} \text{ and } \Q(\sqrt{-11}) \subseteq K, 
    \end{cases}\\
    \mathcal{N}^{22}(W,K,B)\asymp& 
\begin{cases}
    B &\text{ if } W\in\{\langle w_2\rangle, \langle w_{22} \rangle,  \AL(\cX^{22})\} \\
   &\text{ or if } W\in\{\{\id\}, \langle w_{11}\rangle\} \text{ and } \Q(\sqrt{-11}) \subseteq K. 
  \end{cases}
\end{align*}
These bounds are sharp if $K=\Q$ or if both $D\in\{6,10\}$ and $W=\{\id\}$.
\end{theorem}

\begin{remark}
    We expect the bounds in \cref{main_thm} to be sharp in all cases by a conjecture of Loughran and Smeets \cite[Conjecture 1.6]{LS16}.
\end{remark}

\begin{remark}
    \cref{main_thm} makes no statement about counts if $W \in \{\{\textnormal{id}\},\langle w_{11}\rangle\}$ and $K$ does not contain $\mathbb{Q}(\sqrt{-11})$ in the $D=22$ case. 
    See the proof of \cref{main_thm_22} for details. 
\end{remark}

The conditions for having rational points on the top curves $\cX^{6}$ and $\cX^{10}$ (that is, in the $W = \{\id\}$ cases) seen above are to be expected; by a result of Shimura \cite[Theorem 0]{Sh75}, the curve $\cX^{D}$ has no real points for $D>1$.
Furthermore, a result of Jordan \cite[Theorem 2.1.3]{Jor} states that an abelian surface with QM by quaternion discriminant $D$ and with no non-trivial automorphisms has a model over a field $L$ containing its field of moduli if and only if $L$ splits the quaternion algebra over $\Q$ of discriminant $D$. The non-trivial quotients of these curves, on the other hand, are not pointless over $\Q$; for $D \in \{6,10,22\}$ and for each non-trivial $W \leq \AL(\cX^{D})$, the coarse moduli space $X^D/W$ of $\cX^D/W$ is isomorphic to $\mathbb{P}^1_{\Q}$. 

A novel feature of our study comes from the fact that a PQM abelian surface need not have a model defined over its field of moduli. 
This does not occur for elliptic curves, where the field of moduli of an elliptic curve is always a field of definition, as determined by Deligne--Rapoport \cite[\S VI, Proposition 3.2]{DR73}. The obstruction to a genus $2$ curve corresponding to a PQM abelian surface having a model over its field of moduli is explicitly given by Mestre as equivalent to the splitting of a specified conic. This obstruction leads us to consider the problem of counting the number of fibers of certain conic bundles that contain a rational point. To address this, we adapt general counting results of Loughran--Smeets \cite{LS16} and Loughran--Matthiesen \cite{LM24} (see \cref{prop:conic_bundle_count}). 

Our paper is structured as follows. 
In \cref{shim_sec}, we give a brief overview of basic facts about Shimura curves and their Atkin--Lehner quotients. 
In \cref{sec:heights}, we define our height functions, prove an upper bound for counting $D_4$-octic extensions (\cref{proposition:D4}), and prove a result for finding asymptotics for the product of sets (\cref{lem:product}).
In \cref{conics_sec}, we  prove \cref{prop:conic_bundle_count} which will allow us to account for the Mestre obstruction in our point counts. 
The proof of \cref{main_thm} then comes in three parts: \cref{main_thm_6} and \cref{main_thm_6_K} in \cref{6_sec} for the $D=6$ case; \cref{main_thm_10} and \cref{main_thm_10_K} in \cref{10_sec} for the $D=10$ case; and \cref{main_thm_22} and \cref{main_thm_22_K} in \cref{22_sec} for the $D=22$ case.

\section*{Acknowledgments}

This project came out of the 2023 American Mathematical Society Mathematics Research Community (MRC) on Explicit Computations with Stacks. The authors are extremely grateful for having the opportunity to be a part of this community. Many thanks to all the helpers and organizers of the MRC for putting on a wonderful workshop. Special thanks go to John Voight for his guidance and encouragement on this problem, as well as providing thoughtful feedback on an earlier draft of this paper. Thanks also to Arul Shankar for a helpful conversation about counting $D_4$-octic extensions of number fields.

Support for this project was provided by the National Science Foundation under Grant Number 1916439 for the 2023 summer AMS MRC.
During this project T.P.\@ was supported by National Science Foundation grant DMS-2303011. T.S.\@ was supported during the project by FWO Vlaanderen and the Herschel Smith Fund. J.Y.\@ was supported by National Science Foundation RTG grant DMS-2231565.



\section{Preliminaries on Shimura curves}\label{shim_sec}

\subsection{Shimura curves}
Let $B$ be a quaternion algebra over $\Q$ of discriminant $D$. Suppose that $B$ is indefinite, meaning that $B \otimes_{\mathbb{Q}} \mathbb{R} \cong M_2(\mathbb{R})$. It follows that $D=p_1\cdot p_2\cdots p_{2r}$ is a product of an even number of rational primes, namely those at which $B$ is ramified, and that $B$ is unique up to isomorphism. Let $\mathcal{O}$ be a maximal order in $B$ (unique up to conjugacy in $B^\times$).  

Let $\bar{\ }: B\to B$ denote the \textbf{standard involution} on $B$ (i.e., the unique involution for which $\overline{b}b\in \Q$ for all $b\in B$). A \textbf{principal polarization} on $B$ is an element of $\cO$ satisfying $\mu^2=-D$. 
To $\mu$ we associate an involution
\begin{align*}
    \rho_\mu: B &\to B\\
    b &\mapsto \mu^{-1}\overline{b}\mu,
\end{align*}
and further we associate to $\mu$ the moduli problem of classifying triples $(A,\lambda, \iota)$, where
\begin{itemize}
    \item $A$ is an abelian surface,
    \item $\lambda$ is a principal polarization on $A$, and
    \item $\iota:\cO\hookrightarrow \text{End}(A)$ is a ring homomorphism such that $\rho_\mu$ and the Rosati involution $r_{\lambda}:\text{End}(A)\otimes_{\Z}\Q\to\text{End}(A)\otimes_{\Z}\Q$ corresponding to $\lambda$
     coincide on the image of $\iota$ (i.e., if $a\in \iota(\cO)$, then $\rho_\mu(a)=r_\lambda(a)$). 
\end{itemize}
Given such a triple $(A,\lambda,\iota)$, we refer to $\iota$ as a \textbf{quaternionic multiplication (QM) structure} by $(\cO,\mu)$ on the principally polarized abelian surface $(A,\lambda)$. We refer to the triple $(A,\lambda,\iota)$ as a \textbf{(principally polarized) $(\cO,\mu)$-QM abelian surface}.

Two $(\cO,\mu)$-QM abelian surfaces $(A_1,\lambda_1,\iota_1)$ and $(A_2,\lambda_2,\iota_2)$ are said to be isomorphic if there exists an isomorphism $\varphi:A_1\xrightarrow{\sim} A_2$ for which $\varphi^\ast(\lambda_2)=\lambda_1$ and for each $c\in\cO$ we have $\varphi\circ\iota_1(c)=\iota_2(c)\circ \varphi$.

The above moduli problem can be represented by a stack over $\mathbb{Q}$ (for this we refer to \cite[\S 12]{Boutot79}; see also \cite[Main Theorem 1]{Sh67} for the statement on the canonical model of the coarse moduli space), which we denote by $\cX^D_\mu$.

\subsection{Atkin--Lehner quotients and forgetful maps}

Consider the subgroup of the normalizer of $\cO$ in $B^\times$ defined by 
\[
N_{{B}^\times_{> 0}}(\mathcal{O}) \defeq \{u \in {B}^\times : \text{nrd}(u) > 0 \text{ and } u^{-1} \mathcal{O} u = \mathcal{O} \}.
\]
This group acts on the moduli space $\cX^D_\mu$ as follows:
\begin{align*}
N_{{B}^\times_{> 0}}(\mathcal{O}) \times \cX_\mu^D &\to \cX_\mu^D\\
(u,[(A,\lambda,\iota)])&\mapsto [(A,\lambda_u,\iota_u)],
\end{align*}
where $\lambda_u=\iota(u)^{\ast}(\lambda)/\text{nrd}(u)$ and where
\begin{align*}
    \iota_u: \cO&\hookrightarrow \text{End}(A)\\
    a &\mapsto u^{-1}\iota(a)u.
\end{align*}
Note that if $u\in N_{{B}^\times_{> 0}}(\cO)\cap \bbQ^\times$, then $(A,\lambda,\iota)$ and $(A,\lambda_u,\iota_u)$ are isomorphic. 

We define the \textbf{(positive) Atkin--Lehner group} as
\[ 
\AL(\cX^{D}) \defeq N_{{B}^\times_{> 0}}(\mathcal{O}) / \mathbb{Q}^\times \mathcal{O}^\times \subseteq \text{Aut}(\cX^D_\mu).
\]
 This is a finite abelian $2$-group. Recalling that $D=p_1\cdot p_2\cdots p_{2r}$ is a product of distinct primes, it is known that $\AL(\cX^{D}) \cong \left(\mathbb{Z}/2\mathbb{Z}\right)^{2r}$ is generated by involutions $w_{p_i}$ for $1 \leq i \leq 2r$. 

Let $\cA_2$ denote the moduli stack of principally polarized abelian varieties. Then we have the forgetful map
\begin{align*}
    F:\cX^{D}_\mu&\to \cA_2\\
    [(A,\lambda,\iota)]&\mapsto [(A,\lambda)].
\end{align*}
Let $\cV_\mu^D$ denote the image of $F$. At the level of coarse moduli schemes, Rotger has shown that the image $F(X_\mu^D)$ of $X_\mu^D$ under the forgetful map to $A_2$ is birational to a specific Atkin--Lehner quotient of $X_\mu^D$ \cite[Theorem 4.3]{Rot04a}. 

\subsection{Genus two curves}\label{subsec:genus_two_curves}

Let $\cM_2$ denote the moduli stack of genus two curves. The Torelli map,  $\tau:\cM_2\rightarrow \cA_2$ sending a genus two curve to its principally polarized Jacobian, is an open immersion of stacks \cite{OS80}. Let $\cC^D_{\mu}$ denote the preimage of $\cV^D_{\mu}\cap \tau(\cM_2)$ with respect to the Torelli map. We have the following diagram:
\[
\begin{tikzcd}
& \cA_2 & \cM_2 \arrow[l,"\tau" above] \\
\cX^D_\mu \arrow[r, two heads] \arrow[ru, "F"] & \cV^D_\mu \arrow[u] & \cC^D_\mu. \arrow[l,"\tau" above] \arrow[u] \\
\end{tikzcd}
\]

The \textbf{Igusa invariants} of a genus two curve $C$, introduced by Igusa in \cite{Igusa}, are homogeneous polynomials $I_i(C)$ of degree $i$ in the coefficients of a hyperelliptic model of $C$ for $i\in \{2,4,6,10\}$. These invariants define a map
\begin{align*}
 \cM_2(\overline{K})&\to \cP(2,4,6,10)(\overline{K})\\
    C &\mapsto [I_2(C): I_4(C):I_6(C):I_{10}(C)].
\end{align*}
This restricts to a map 
\begin{align*}
 \cM_2(K)\rightarrow \cP(2,4,6,10)(\overline{K}).
\end{align*}
The Igusa invariants can be normalized by rigidifying along the $C_2$-gerbe of $\cP(2,4,6,10)$ (see, e.g., \cite[Appendix C]{AGV08} for details about rigidification). In particular, let
\[
\sI:\cM_2\to\cP(1,2,3,5)
\]
be the morphism (defined over $\overline{K}$) for which the following diagram commutes: 
\[
\begin{tikzcd}
\cM_2 \arrow[r] \arrow[rd,swap, "\sI"] & \cP(2,4,6,10) \arrow{d}{\text{$C_2$--rigidifiction}} \\
 & \cP(1,2,3,5).
\end{tikzcd}
\]

Mestre has shown that a point $x\in\cP(1,2,3,5)(K)$ is contained in the image of $\sI$ if and only if a certain conic (depending on $x$) has a rational point over $K$ \cite{Mestre}. This obstruction for a set of $K$-rational Igusa invariants to correspond to a genus two curve defined over $K$ is referred to as the \textbf{Mestre obstruction}. The conic that arises is referred to as the \textbf{Mestre conic} corresponding to $x$.

\subsection{The $\mathcal{O}$-PQM locus}

The image $\cV_\mu^D$ of $\cX_\mu^D$ under the forgetful map $F$ is contained in the \textbf{$\mathcal{O}$-PQM locus of $\mathcal{A}_2$}, which we denote by $Q_{\mathcal{O}} \subseteq \mathcal{A}_2$. This is the locus whose points over a number field $K$ consist of geometric isomorphism classes $[(A,\lambda)]$ of $\cO$-PQM abelian surfaces with field of moduli $K$. Naturally, one might ask for the full structure of this locus. One must vary the polarization $\mu$ of the quaternion order $\mathcal{O}$, and at the level of associated coarse moduli schemes this is the entire picture. This is stated in the following proposition of Rotger, which we state in a simplified version sufficient for our needs. We let $X^D_\mu$ denote the coarse moduli scheme over $\mathbb{Q}$ associated to $\cX^D_\mu$, and let $V^D_{\mu}$ denote that associated to $\cV^D_\mu$

\begin{prop}\cite[Proposition 5.3]{Rot04b}
\begin{itemize}
    \item The locus $Q_\mathcal{O}$ is the set $\mathcal{Q}_\mathcal{O}(\mathbb{C})$ of complex points of a reduced, complete subscheme $\mathcal{Q}_\mathcal{O}$ of $A_2$ defined over $\mathbb{Q}$. 
    \item There exist finitely many principal polarizations $\mu_1, \ldots, \mu_k$ of $\mathcal{O}$ such that 
    \[ \mathcal{Q}_\mathcal{O} = \bigcup_{i = 1}^{k} V^D_{\mu_i} \]
    is the decomposition of $\mathcal{Q}_\mathcal{O}$ into irreducible components. 
\end{itemize}
\end{prop}

\section{Heights}\label{sec:heights}

Let $\Val(K)$ denote the set of places of the number field $K$, and let $\Val_0(K)$ (resp.\@ $\Val_\infty(K)$) denote the finite (resp.\@ infinite) places of $K$.
For each finite place $v\in \Val_0(K)$,  let $\pi_v$ be a uniformizer for the completion $K_v$ of $K$ at $v$.
Let $\w$ denote the $(n+1)$-tuple  $(w_0,w_1,\dots,w_n)$ of positive integers. 
For $x=(x_0,\dots,x_{n})\in K^{n+1}- \{(0,0,\ldots,0)\}$ define
\[
 |x|_{{\w},v}\defeq\begin{cases}
 \max_i\left\{|\pi_v|_v^{\lfloor v(x_i)/w_i\rfloor}\right\} & \text{ if } v\in \Val_0(K),\\
 \max_i\left\{|x_i|_v^{1/w_i}\right\} & \text{ if } v\in \Val_\infty(K).
 \end{cases}
 \]
 We define the \textbf{height} of a point $x=[x_0:\cdots:x_n]\in\cP(\w)(K)$ to be
\[
\Ht_\w(x)\defeq\prod_{v\in \Val(K)} |(x_0,\dots,x_n)|_{\w,v}.
\]
 Define the \textbf{Igusa height} on $\cM_2(K)$ as
 \begin{align*}
    \overline{\Ht}: \cM_2(K) &\to \bbR_{>0}\\
    C &\mapsto \min\left\{\Ht_{(1,2,3,5)}(\sI(C')) : 
    C'\cong C \text{ over } \overline{K}\right\}.
 \end{align*}
 Note that for curves $C\in \cM_2(K)$ for which $\sI(C)$ is (geometrically) one of the three stacky points of $\cP(1,2,3,5)$ we have $\overline{\Ht}(C)=1$.
  For all other curves
  \[
  \overline{\Ht}(C)=\Ht_{(1,2,3,5)}(\sI(C)).
  \]
  The Igusa height only depends on the $\overline{K}$-isomorphism class of the curve. 

  We now define a height that can be used to count $K$-isomorphism classes of genus $2$ curves. For a genus two curve $C$ over $K$ let $\mathfrak{N}(C)$ denote the conductor of the Jacobian of $C$ over $K$.
  For each twist class of genus two curves choose a genus two curve $C_{\min}$ over $K$ for which 
  \[
  N_{K/\Q}(\mathfrak{N}(C_{\min})) =\min\left\{N_{K/\Q}(\mathfrak{N}(C')) : C'\cong C \text{ over } \overline{K}\right\}.
  \]

  Let $K_C$ denote a finite extension (of minimal degree) of $K$ for which $C\cong C_{\min}$ over $K_C$. 
  
Define the \textbf{height} of a genus two curve $C\in \cM_2(K)$ as
\begin{equation}\label{eq:height}
\Ht(C)\defeq \overline{\Ht}(C) \cdot |N_{K/\Q}(\disc_K(K_C))|.
\end{equation}

\begin{remark}
    This height will depend upon the choice of twist minimal curves $C_{\min}$. It would be interesting to describe a canonical way of choosing these curves. However, this choice will not affect the asymptotic results of this paper. 
\end{remark}

\subsection{Counting genus two curves in twist classes}

Let $C$ be a genus two curve over the number field $K$. The possible geometric automorphism groups $\Aut_{\overline{K}}(C)$ of $C$ are the cyclic group of order two $C_2$, the Klein four group $C_2\times C_2$, the dihedral group of order eight $D_4$, the cyclic group of order ten $C_{10}$, the dihedral group of order twelve $D_6$, the special linear group $\SL_2(\bbF_3)$, and the general linear group $\GL_2(\bbF_3)$. However, the only geometric automorphism groups we will encounter in this paper are $C_2$, $C_2\times C_2$, and $D_4$. 

We have defined our height in (\ref{eq:height}) in such a way that counting genus two curves in the twist class of a genus two curve $C$ is equivalent to counting field extensions of $K$ with Galois group $\Aut_{\overline{K}}(C)$ by discriminant. Malle's conjecture (and its refinements) predict asymptotics for counting field extensions by discriminant. For example, the following Proposition is an immediate consequence of Wright's work on enumerating abelian extensions of number fields \cite{Wri89}.

\begin{prop}\label{proposition:Wright}
    Let $C$ be a genus two curve over $K$. Then
    \begin{align*}
    \#&\{C'\in \cM_2(K) : \Ht(C')\leq B \text{ and } C'\cong C \text{ over } \overline{K}\}\\
    &\qquad \asymp
    \begin{cases}
        B & \text{ if } \Aut_{\overline{K}}(C)\cong C_2,\\
        B^{1/2}\log(B)^2 & \text{ if } \Aut_{\overline{K}}(C)\cong C_2\times C_2.
    \end{cases}
    \end{align*}
\end{prop}

Malle's conjecture, in the case of $D_4$-octic extensions, leads to the following prediction.

\begin{conj}\label{conjecture:Malle_D4}
    Let $C$ be a genus two curve over $K$ with $\Aut_{\overline{K}}(C)\cong D_4$. Then
    \begin{align*}
    \#&\{C'\in \cM_2(K) : \Ht(C')\leq B \text{ and } C'\cong C \text{ over } \overline{K}\}\sim B^{1/4}\log(B)^2.
    \end{align*}
\end{conj}

We prove a weak bound for counting $D_4$-octics.

\begin{prop}\label{proposition:D4}
    Let $C$ be a genus two curve over $K$ with $\Aut_{\overline{K}}(C)\cong D_4$. Then
    \begin{align*}
    \#&\{C'\in \cM_2(K) : \Ht(C')\leq B \text{ and } C'\cong C \text{ over } \overline{K}\}= O(B^{1/2}).
    \end{align*}
\end{prop}

\begin{proof}
    We want to count $D_4$-octic extensions over $K$ by discriminant. 
    If $L_8/K$ is a $D_4$-octic extension then there exists a $D_4$-quartic extension $L_4$ whose Galois closure $\tilde{L}_4$ is isomorphic to $L_8$. Conversely, the Galois closure of any $D_4$-quartic isomorphic to $L_4$ will be isomorphic to $L_8$. The relationship between the discriminants $\disc_{K}(L_8)$ and $\disc_{K}(L_4)$ (see, e.g., \cite[Chapter III Corollary 2.10]{Neu}) gives the following inequality,
    \begin{equation}
        |N_{K/\Q}(\disc_K(L_8))|\geq |N_{K/\Q}(\disc_K(L_4))|^2.
    \end{equation}
    Therefore, counting $D_4$-quartics whose discriminant is bounded by $B^{1/2}$ will give an upper bound for counting $D_4$-octics with discriminant bounded by $B$.

     Malle's conjecture is known for $D_4$-quartic extensions over number fields; this was stated in \cite{CDyDO02} and can be viewed as a special case of the main result of \cite{Klu12}. In particular, 
     \begin{equation*}
         \#\{L_4/K : |N_{K/\Q}(\disc_K(L_4))|\leq B^{1/2},\ [L_4:K]=4, \text{ and } \Gal(\tilde{L}_4/K)\cong D_4\}\asymp B^{1/2}.
     \end{equation*}
\end{proof}

\subsection{A product lemma}
We now prove a lemma for counting points on products of sets. For complex numbers $z_1,z_2\in \bbC$ with positive real part the \textbf{beta function} is defined as
\[
\Beta(z_1,z_2)\defeq \int_0^1 t^{z_1-1}(1-t)^{z_2-1}\,dt.
\]

\begin{lemma}[Product Lemma]\label{lem:product} 
Let $\Ht_X$ and $\Ht_Y$ be counting functions on sets $X$ and $Y$ respectively. Assume for all $x\in X$ and $y\in Y$ that $\Ht_X(x)\neq 0$ and $\Ht_Y(y)\neq 0$. Suppose that $c_X,c_Y\in \bbR_{\neq 0}$ and $\alpha_X,\alpha_Y, \beta_X,\beta_Y\in \bbQ$ with $\alpha_X\geq\alpha_Y > 0$ and $\beta_Y \geq 0$ are such that
\[
\#\{x\in X: \Ht_X(x)\leq B\} \sim c_X B^{\alpha_X}\log(B)^{\beta_X}
\]
and 
\[
\#\{y\in Y: \Ht_Y(x)\leq B\} \sim c_Y B^{\alpha_Y}\log(B)^{\beta_Y}.
\]
Let 
\begin{align*}
c_1 &\defeq \Beta(\beta_Y+1,\beta_X+1) c_Yc_X\alpha_X \\
c_2&\defeq c_X c_Y \alpha_Y  \sum_{k=1}^{\infty} \frac{\log(k)^{\beta_Y}}{k^{\alpha_X-\alpha_Y+1}}.
\end{align*}
Then $\#\{(x,y)\in X\times Y: \Ht_X(x)\cdot \Ht_Y(y)\leq B\}$ is asymptotic to
\begin{align*}
 \begin{cases}
c_1  B^{\alpha_X}\log(B)^{\beta_X+\beta_Y+1} & \text{ if } \alpha_X=\alpha_Y, \\
c_2 B^{\alpha_X}\log(B)^{\beta_X} & \text{ if } \alpha_X>\alpha_Y.
\end{cases}
\end{align*}
\end{lemma} 

\begin{remark}
   Similar results to \cref{lem:product} can be found in \cite[Lemma 3.1 and Lemma 3.2]{Wan21}, \cite[Proposition 2]{FMT89}, and \cite[Lemma 7.1]{Den98}. The main difference in our result is that it allows for $\beta_X<0$.
\end{remark}

\begin{proof}
    Let $\lambda_1 \leq \lambda_2 \leq \dots$ be an increasing enumeration of the multiset $\{\Ht_Y(y): y \in X\}$, so that
    \begin{salign}\label{eq:product_lemma_1}
         &\hspace{-1cm} \#\{(x,y)\in X\times Y: \Ht_Y(x)\cdot \Ht_Y(y)\leq B\}\\ 
        &= \sum_{i=1}^{\#\{y\in Y: \Ht_X(y)\leq B\}} \#\{x\in X: \Ht_Y(x)\leq B/\lambda_i\}\\
        &= \sum_{i=1}^{\#\{y\in Y: \Ht_X(y)\leq B\}} c_X(B/\lambda_i)^{\alpha_X}\log(B/\lambda_i)^{\beta_X}+o\left((B/\lambda_i)^{\alpha_X}\log(B/\lambda_i)^{\beta_X}\right).
    \end{salign}
        For $j\in \Z_{\geq 0}$ let $a(j)\defeq \#\{y\in Y: j < \Ht_Y(y) \leq j+1\}$. Then (\ref{eq:product_lemma_1}) equals
    \begin{salign}\label{eq:product_lemma_2}
         & \sum_{j=1}^{B} a(j) c_X (B/j)^{\alpha_X}\log(B/j)^{\beta_X}+o\left(a(j) c_X (B/j)^{\alpha_X}\log(B/j)^{\beta_X}\right)\\
         & \qquad + O\left(a(0)c_Y \left(B/\lambda_1\right)^{\alpha_Y} \log(B/\lambda_1)^{\beta_Y}\right).\\
    \end{salign}
     The first term of (\ref{eq:product_lemma_2}) is
    \begin{equation}\label{eq:product_lemma_3}
        \sum_{j=1}^{B} a(j) c_X (B/j)^{\alpha_X}\log(B/j)^{\beta_X}
        =c_X B^{\alpha_x}\sum_{j=1}^{B} a(j) j^{-\alpha_X}\log\left(B/j\right)^{\beta_X}.
    \end{equation}
  
      Now we separate into two cases. 
        In the case $\alpha_X = \alpha_Y$, we have
        \begin{align*}
            \sum_{j=1}^k a(j)\sim c_Y k^{\alpha_Y}\log(k)^{\beta_Y}=c_Y k^{\alpha_X}\log(k)^{\beta_Y}.
        \end{align*}
        Therefore, using Abel summation, (\ref{eq:product_lemma_3}) is asymptotic to
        \begin{align*}
        -c_X B^{\alpha_X} \int_1^B \left(c_Yu^{\alpha_X}\log(u)^{\beta_Y}\right)\left(u^{-\alpha_X-1}\log(B/u)^{\beta_X}\left(\alpha_X - \frac{\beta_X}{\log(B/u)}\right)\right) du.
        \end{align*}
        Using the fact that $\log(B/u)^{-1}=O(1)$ and rearranging terms, this is asymptotic to
        \begin{align*}
        &c_X c_Y \alpha_X B^{\alpha_X} \int_1^B u^{-1}\log(u)^{\beta_Y}\log(B/u)^{\beta_X} du\\
        &\qquad = c_X c_Y \alpha_X B^{\alpha_X} \int_1^B u^{-1}\log(u)^{\beta_Y}\left(\log(B)-\log(u)\right)^{\beta_X} du.
        \end{align*}
        Making the substitution $v=\log(u)/\log(B)$, this equals
        \begin{align*}
        c_X c_Y \alpha_X B^{\alpha_X}\log(B)^{\alpha_X+\alpha_Y+1} \int_0^1 v^{\beta_Y}\left(1-v\right)^{\beta_X} dv=c_1 B^{\alpha_X}\log(B)^{\beta_X+\beta_Y+1}.
        \end{align*}

    Now consider the case $\alpha_X>\alpha_Y$. Applying Newton's binomial theorem, (\ref{eq:product_lemma_3}) becomes
    \begin{salign}\label{eq:product_lemma_c2_0}
        &\sum_{j=1}^{B} a(j) c_X (B/j)^{\alpha_X}\left(\log(B)-\log(j)\right)^{\beta_X}\\
       &\qquad  = \sum_{j=1}^{B} a(j) c_X (B/j)^{\alpha_X}\left(\log(B)^{\beta_X}-O\left(\log(B)^{\beta_X-1}\log(j)\right)\right).
    \end{salign}
    
    Using summation by parts, the main term of (\ref{eq:product_lemma_c2_0}) is
    \begin{salign}\label{eq:product_lemma_c2_1}
        \sum_{j=1}^{B} a(j) c_X (B/j)^{\alpha_X}\log(B)^{\beta_X}
        &=c_XB^{\alpha_X} \log(B)^{\beta_X} \sum_{k=1}^{B} \left(\frac{1}{k^{\alpha_X}}-\frac{1}{(k-1)^{\alpha_X}}\right)\sum_{j=1}^k a(j)\\
        &\sim c_XB^{\alpha_X} \log(B)^{\beta_X} \sum_{k=1}^{B} \frac{\alpha_X c_Y k^{\alpha_Y} \log(k)^{\beta_Y}}{k^{\alpha_X+1}}\\
        & \sim c_Xc_Y \alpha_X B^{\alpha_X} \log(B)^{\beta_X} \sum_{k=1}^{B} \frac{ \log(k)^{\beta_Y}}{k^{\alpha_X-\alpha_Y+1}}.
    \end{salign}
        Similarly, the error term of (\ref{eq:product_lemma_c2_0}) can be shown to be 
        \begin{equation}\label{eq:product_lemma_c2_2}
            O(B^{\alpha_X} \log(B)^{\beta_X-1})\sum_{k=1}^{B} \frac{ \log(k)^{\beta_Y+1}}{k^{\alpha_X-\alpha_Y+1}}.
        \end{equation}
     Since $\alpha_X>\alpha_Y$, the sum
    \begin{equation*}
        \sum_{k=1}^{B} \frac{ \log(k)^{\beta_Y}}{k^{\alpha_X-\alpha_Y+1}}
    \end{equation*} 
    converges as $B$ tends to infinity. Therefore (\ref{eq:product_lemma_c2_1}) is
    \begin{equation*}
        c_Xc_Y \alpha_X B^{\alpha_X} \log(B)^{\beta_X} \sum_{k=1}^{B} \frac{ \log(k)^{\beta_Y}}{k^{\alpha_X-\alpha_Y+1}} \sim c_2 B^{\alpha_X} \log(B)^{\beta_X} + O(B^{\alpha_Y+\epsilon}).
    \end{equation*} 
    Similarly, the error term (\ref{eq:product_lemma_c2_2}) is 
    \begin{equation*}
        O\left(B^{\alpha_X} \log(B)^{\beta_X-1}\right).
    \end{equation*} 
    These estimates, together with (\ref{eq:product_lemma_2}), prove the lemma in this case.
\end{proof}

\section{Counting fibers of conic bundles which have a rational point}\label{conics_sec}




    Let $\pi:Z\to X$ be a dominant proper morphism of smooth irreducible varieties over a perfect field $k$. For each point $x\in X$ let $k(x)$ denote the residue field of $x$, and choose a finite group $\Gamma_x$ through which the absolute Galois group $\Gal(\overline{k(x)}/k(x))$ acts on the irreducible components of 
    \[
    \pi^{-1}(x)_{\overline{k(x)}}\defeq \pi^{-1}(x)\otimes_{k(x)}\overline{k(x)}.
    \]
    Define $\delta_x(\pi)$ to be the quantity
\[
    \delta_x(\pi)\defeq \frac{\#\{\gamma\in \Gamma_x : \gamma \textnormal{ \small fixes an irred. component of $\pi^{-1}(x)_{\overline{k(x)}}$ of multiplicity $1$}\}}{\#\Gamma_x}.
    \]
    Let $X^{(1)}$ denote the set of codimension $1$ points of $X$, and define
    \[
    \Delta(\pi) \defeq \sum_{D\in X^{(1)}} (1-\delta_D(\pi)).
    \]

\begin{prop}\label{prop:conic_bundle_count}
    Let 
    \begin{align*}
\varphi:\mathbb{P}^1&\to \mathcal{P}(w_0, w_1,\dots,w_n)\\
    [x:y]&\mapsto [f_0(x,y):f_1(x,y): \cdots : f_n(x,y)]
    \end{align*}
    be a morphism of stacks over $K$ for 
    which there exists a $\delta\in \Z_{>0}$ such that 
    the $f_i$ are homogeneous of degree $\delta w_i$. Let
    \[
    \pi: Z\to \mathbb{P}^1
    \]
    be a non-singular conic bundle over $K$, all of whose non-split fibers lie above rational points. Assume there exists a smooth fiber of $\pi$ with a rational point. Then
    \[
    \#\{x\in \mathbb{P}^1(K) : \Ht_{(w_0,\dots,w_n)}(\varphi(x))\leq B,\ x\in \pi(Z(K))\}\ll \frac{B^{2/\delta}}{\log(B)^{\Delta(\pi)}}.
    \]
    If $K=\Q$ or $\Delta(\pi)=0$ then this bound is sharp.
\end{prop}

\begin{proof}
    As $\deg(f_i)=\delta w_i$, we find that $\varphi^{\ast}\mathcal{O}_{\mathcal{P}(w_0,\dots,w_n)}(1)$ is linearly equivalent  to $\mathcal{O}_{\mathbb{P}^1}(\delta)$ as line bundles on $\mathbb{P}^1$. Therefore,
    \[
    \Ht_{(w_0,\dots,w_n)}(\varphi(x))=\Ht_{\varphi^{\ast}\mathcal{O}(\delta)}(x)=\Ht_{(1,1)}(x^\delta)+O(1),
    \]
    where the error term $O(1)$ accounts for the possible difference of height functions (i.e., the possibly different metrizations of $\varphi^{\ast}\mathcal{O}_{\mathcal{P}(w_0,\dots,w_n)}(1)$ and $\mathcal{O}_{\mathbb{P}^1}(\delta)$).
    
    Since the Hasse principle holds for conics, the desired result follows from 
    \cite[Theorem 1.2]{LS16}. When $\Delta(\pi)=0$ this is sharp by \cite[Theorem 1.3]{LS16}. That this is sharp in the case $K=\Q$ follows from \cite[Corollary 1.2]{LM24}. 
\end{proof}


 \section{The cases of discriminant $6$, $10$, and $22$}

In this section we restrict to the cases of discriminant $D\in \{6,10,22\}$. Main references for results used about the corresponding Shimura curves $\cX^{D}_\mu$ and their quotients are \cite{BG08} and \cite{NR15}. 

\subsection{The $\mathcal{O}$-PQM locus for $D\in\{6,10,22\}$}
The Shimura curves $\cX^D_\mu$ of genus $0$ with $D>1$ are those with $D \in \{6,10,22\}$.

It is known that, up to canonical isomorphism, the moduli stack $\cX^D_\mu$ is independent of the polarization $\mu$ (see, e.g., \cite[Lemma 43.6.23]{Voight} and \cite[\S 6]{Rot04b}). For this reason, we will suppress $\mu$ in our notations related to these curves and their images under the corresponding forgetful maps.

A result of Rotger \cite[Theorem 8.1]{Rot04b} states that, for $D \in \{6,10\}$, the $\mathcal{O}$-PQM locus $\mathcal{Q}_{\mathcal{O}}$ is irreducible, hence equal to $V_\mu^D$, and is birational to the full Atkin--Lehner quotient $X^D_\mu/\AL(X^D_\mu)$. The same follows in the $D=22$ case from Rotger's result since the locus is irreducible and this is a twisting case, as the quaternion algebra of discriminant $22$ is isomorphic to $\left( \dfrac{-22,2}{\Q} \right)$.
This justifies viewing the counts of points on $\cX^D/\AL(\cX^D)$ in these cases as counts of points on the corresponding $\mathcal{O}$-PQM loci of $\cA_2$.  

\subsection{Canonical rings}\label{can_ring_sec}
For $D\in \{6,10,22\}$ and $W \leq \AL(\cX^D)$ a subgroup of the Atkin--Lehner group, set $\cX^D_{W}\defeq \cX^D/W$ and let $X^{D}_{W}$ denote the coarse space of $\cX^D_{W}$. Let $\Omega$ denote the canonical bundle of the quotient stack $ \cX^D_W$, and let
\[
R^D_{W}\defeq \bigoplus_{d=0}^\infty H^0\left(\cX^D_W, \Omega^{\otimes d}\right)
\]
denote the canonical ring of $\cX^D_{W}$. 

Define $\cC^D_{W}$ analogously to how $\cC^D$ was defined in \cref{subsec:genus_two_curves}, but with $\cX^D=\cX^D_{\{\id\}}$ replaced by $\cX^D_{W}$. Let $C_{W}^D$ denote the coarse space of $\cC^D_{W}$.

\subsection{The discriminant $6$ case}\label{6_sec}

It follows from \cite[Proposition 3.2]{BG08} that the canonical ring of $\cX^6$ is
\begin{align}\label{eq:R_id}
R^6_{\{\id\}}\cong\Q[h_4,h_6,h_{12}] / (h_{12}^2+3h_6^4 + h_4^6).
\end{align}

The Atkin--Lehner group acts on the modular forms $h_4$, $h_6$, and $h_{12}$ as follows:
\begin{align*}
    h_4&=-w_2(h_4)=-w_3(h_4)=w_6(h_4)\\
    h_6&=w_2(h_6)=-w_3(h_6)=-w_6(h_6)\\
    h_{12}&=-w_2(h_{12})=w_3(h_{12})=-w_6(h_{12}).
\end{align*}

The ring of invariants of $R^6_{\{\textnormal{id}\}}$ with respect to the action of the involution $w_6= w_2\circ w_3\in \AL(\cX^6)$ is generated by $h_4, h_6^2, h_6h_{12}$, with the relation $((h_6h_{12})^2 + 3(h_6^2)^3 + h_4^6h_6^2=0$. Setting $x=h_4$, $y=h_6^2$, and $z=h_6h_{12}$, we find that
\begin{align}\label{eq:R_w6}
\begin{split}
R^6_{\langle w_6\rangle} &\cong \mathbb{Q}[h_4,h_6^2,h_6h_{12}]/(h_{6}^2h_{12}^2+3h_6^6+h_4^6h_6^2)\\
&\cong \mathbb{Q}[x,y,z]/(z^2 + 3y^3 + x^6y).
\end{split}
\end{align}

Similarly, we compute
\begin{align}\label{eq:canonical_rings6}
\begin{split}
R^6_{\langle w_3\rangle} &\cong \mathbb{Q}[h_4h_6,h_4^2,h_6^2,h_{12}] / \left( {h_{12}}^2+3h_6^4+h_4^6, (h_4h_6)^2-h_4^2h_6^2 \right) \\
&\cong \Q[w,x,y,z]/\left( z^2+3y^2+x^3, w^2-xy\right)\\
R^6_{\langle w_2 \rangle} &\cong  \mathbb{Q}[h_4^2,h_6,h_4h_{12}]/\left( (h_4h_{12})^2 + 3h_4^2h_6^4 + (h_4^2)^4 \right) \\
&\cong \Q[x,y,z]/\left(z^2+3xy^4+x^4\right)\\
R^6_{\AL(\cX^6)} &\cong \mathbb{Q}[h_4^2,h_6^2,h_4h_6h_{12}]/\left( \left(h_4h_6h_{12}\right)^2 + 3h_4^2(h_6^2)^3 + (h_4^2)^4h_6^2 \right)\\
&\cong \Q[x,y,z]/\left(z^2+3xy^3+x^4y\right). 
\end{split}
\end{align}

 From \cite[Proposition 3.6]{BG08} we know that the weight zero modular form $j=\left(4h_6^2/3h_4^3\right)^2$ on $\cX^6$ induces an isomorphism $j:C^6\to\mathbb{P}^1-\{0,\infty\}$. Taking Igusa invariants of the curve corresponding to $j\in \mathbb{P}^1-\{0,\infty\}$ induces a morphism
\begin{align*}
\sJ: \mathbb{P}^1-\{0,\infty\} &\to \cP(1,2,3,5)\\
j &\mapsto [12(j+1) : 6(j^2+j+1) : 4(j^3-2j^2+1) : j^3]. 
\end{align*}
We now give a description of the Mestre obstruction in each case.

\begin{prop}\label{prop:Mestre_obstructions_6}
  For each subgroup $W$ of the Atkin--Lehner group $\AL(\cX^6)$, outside of finitely many $j$, the Mestre obstruction on $C^6_{W}$ can be described by the Hilbert symbol given in \cref{tab:Mestre_obstruction}.
\end{prop}

\begin{table}[h]
\centering
\renewcommand{\arraystretch}{1.2}
\begin{tabular}{cc}
\toprule
$W$ & Mestre obstruction\\
\midrule
$\{\id\}$ & $(-6,2)_K$ \\
$\langle w_2 \rangle$ & $(-6x,2)_K$\\
$\langle w_3 \rangle$ & $(-6x,2x)_K$\\
$\langle w_6\rangle$ & $(-6,-2y)_K$\\
$\AL(\cX^6)$ & $(-6x,2y)_K$\\
\bottomrule
\end{tabular}
\caption{Mestre obstructions on $C_W^{6}$.}
\label{tab:Mestre_obstruction}
\end{table}

\begin{proof}
    By \cite[Proposition 3.12 and Proposition 3.13]{BG08}, the Mestre obstruction to a curve $C$, corresponding to a point of $C^6$, with rational Igusa invariants being defined over a number field $K$ can be described by the Hilbert symbol
\begin{equation}\label{eq:Mestre_6}
(-6j, -2(27j+16))_K
\end{equation}
outside of the finitely many $j$ such that the corresponding curve has extra automorphisms (specifically, $j = -16/27$ corresponding to the genus $2$ curve $C_{(-24)}$ with $-24$-CM and $j = 81/64$ corresponding to the genus $2$ curve $C_{(-19)}$ with $-19$-CM from \cite[\S 3.3]{BG08}).

 By our expression (\ref{eq:R_id}) for the canonical ring $R^6_{\{\id\}}$ we have $h_{12}^2+3h_6^4 + h_4^6=0$. Also, we have by definition that $j=16h_6^4/9h_4^6$.  Therefore, we can rewrite the Mestre obstruction as follows:
\begin{align*}
    \left(-6\frac{16h_6^4}{9h_4^6}, -2\left(27\frac{16h_6^4}{9h_4^6}+16\right)\right)_K
    &=(-6(16h_6^4), -2(27(16h_6^4)-16(9h_4^6))_K\\
    &=(-6,-2(3h_6^4 + h_4^6))_K\\
    &=(-6,2 h_{12}^2)_K\\
    &=(-6,2)_K.
\end{align*}

The other cases are similar, starting with (\ref{eq:Mestre_6}) and using the relations from the expressions (\ref{eq:R_w6}) and (\ref{eq:canonical_rings6}) for the canonical rings.
\end{proof}


Let $\pi:\mathcal{B}\to\mathbb{P}^1_j$ be the conic bundle corresponding to the Mestre obstruction on $C_{W}^6$. We have the following commutative diagram:

\begin{equation}\label{diagram:commutative_triangle}
\begin{tikzcd}
C^6_{W}(K) \arrow{r}{\sI} \arrow[rd, swap,
"j"] & \cP(1,2,3,5)(K) \\
 & \mathbb{P}_j^1(K)-\{0,\infty\} \arrow[swap]{u}{\sJ}\\
 & \mathcal{B.} \arrow{u}{\pi}
\end{tikzcd}
\end{equation}
Recall from the introduction that we define the function $\overline{\mathcal{N}}^6(W,K,B)$ to be that which counts, up to isomorphism over $\overline{K}$, curves $C$ in $C_{W}^6(K)$ which have a model over $K$ and satisfy $\overline{\Ht}(C)\leq B$.

\begin{theorem}\label{main_thm_6}
We have that 
\[
\mathcal{\overline{N}}^6(W,K,B)\ll
\begin{cases}
   B^2/\log(B) &\text{ if } W\in\{\AL(\cX^6),\langle w_3\rangle\}, \\
    B^2/\log(B)^{1/2} &\text{ if } W\in\{\langle w_2\rangle, \langle w_6\rangle\},\\
    B^2 &\text{ if } W=\{\textnormal{id}\} \text{ and } (-6,2)_K=1, \\
    0 &\text{ if } W=\{\textnormal{id}\} \text{ and } (-6,2)_K=-1.
\end{cases}
\]
These bounds are sharp when $K=\Q$ or $W=\{\id\}$.
\end{theorem}

\begin{proof}
      For $W\neq \{\textnormal{id}\}$ we apply \cref{prop:conic_bundle_count} with $\varphi=\sJ$. By the definition of $\sJ$ we see that $\delta=1$. Although the Mestre obstructions computed in \cref{prop:Mestre_obstructions_6} ignore finitely many points in $\mathbb{P}^1_j$ (i.e., finitely many curves in $C^{6}_{W}$), this will not affect our asymptotic results. We now compute $\Delta(\pi)$.  Note that the smooth fibers of $\pi$ will only have a single irreducible component, and thus $\delta_x(\pi)=1$ in these cases. It remains to study the geometrically reducible fibers. From \cref{prop:Mestre_obstructions_6} we have that the conic bundle $\pi:\mathcal{B}\to \mathbb{P}^1_j$ has exactly two reducible fibers when $W=\AL(\cX^{6})$ and exactly one reducible fiber when $W\in\{\langle w_2\rangle, \langle w_6\rangle\}$. In each of these cases the reducible fiber $\pi^{-1}(x)_{\overline{K}}$ has two irreducible components and $\Gal(\overline{K}/K)$ acts transitively on these components, so that $\delta_x(\pi)=1/2$. In the case $W=\langle w_3\rangle$ we have a single reducible component of multiplicity $2$. Therefore $\delta_x(\pi)=0$. It follows that 
    \[
    \Delta(\pi)=\begin{cases}
        1 \text{ if } W\in \{\AL(\cX^{6}),\langle w_3\rangle\},\\
        1/2 \text{ if } W\in\{\langle w_2\rangle, \langle w_6\rangle\}.
    \end{cases}
    \]
    Applying \cref{prop:conic_bundle_count} gives the desired result.

      When $W=\{\id\}$ we either have $(-6,2)_K=1$, in which case the Mestre obstruction vanishes and the desired result follows from \cref{prop:conic_bundle_count}, or we have $(-6,2)_K=-1$, in which case $C^6_{\AL(\cX^6)}(K)=\emptyset$. 
\end{proof}

We now bound the number of $K$-isomorphism classes.

\begin{theorem}\label{main_thm_6_K}
We have that 
\[
\mathcal{N}^6(W,K,B)\ll
\begin{cases}
   B^2/\log(B) &\text{ if } W\in\{\AL(\cX^6),\langle w_3\rangle\}, \\
    B^2/\log(B)^{1/2} &\text{ if } W\in\{\langle w_2\rangle, \langle w_6\rangle\},\\
    B^2 &\text{ if } W=\{\textnormal{id}\} \text{ and } (-6,2)_K=1, \\
    0 &\text{ if } W=\{\textnormal{id}\} \text{ and } (-6,2)_K=-1.
\end{cases}
\]
These bounds are sharp when $K=\bbQ$ or $W=\{\id\}$.
\end{theorem}

\begin{proof}
    From \cite[Proposition 3.12]{BG08} there are finitely many curves corresponding to points of $C^6_W$ with geometric automorphism group $C_2\times C_2$, and outside of these exceptions the geometric automorphism group is $C_2$. From \cref{proposition:Wright} the contribution to $\cN^6(W,K,B)$ from curves with $\Aut_{\overline{K}}(C)\cong C_2\times C_2$ is $O(B)$. From \cref{proposition:Wright}, \cref{main_thm_6}, and \cref{lem:product}, the contribution from curves with $\Aut_{\overline{K}}(C)\cong C_2$ gives the bounds in the statement of the theorem.
\end{proof}

\subsection{The discriminant $10$ case}\label{10_sec}

We first describe the stacky geometry of the Atkin--Lehner quotients $\cX^{10}_{W}$. For each $W$ the stack $\cX^{10}_{W}$ has a generic $\mu_2$-gerbe. Rigidifying along this gerbe we get a stacky curve  $\cX^{10}_{W}\mathbin{\!\!\pmb{\fatslash}} \mu_2$. Let $v_i$ denote the number of $\mu_i$-stacky points of $\cX^{10}_{W}\mathbin{\!\!\pmb{\fatslash}} \mu_2$, and set $s_k\defeq \dim\left(H^0\left(\cX^{10}_{W},\Omega^{\otimes k}\right)\right)$. In this case we will only have $\mu_2$-stacky points and $\mu_3$-stacky points. As we are in the genus $0$ case, the dimension $s_k$ can be written in terms of $v_2$ and $v_3$ as follows for even weight $k > 2$ \cite[Theorem 2.24]{Sh_book}:
\begin{equation}\label{cusp_dimension_formula}
s_k=1-k + v_2 \left\lfloor \frac{k}{4} \right\rfloor + v_3 \left\lfloor \frac{k}{3} \right\rfloor + v_4 \left\lfloor \frac{3k}{8} \right\rfloor + v_6 \left\lfloor \frac{5k}{12} \right\rfloor,
\end{equation}
while $s_2 = 0$. The curve $\cX^{10}$ has $4$ points with CM by discriminant $-3$ and has no $-4$-CM points, as $\mathbb{Q}(i)$ does not split the quaternion algebra over $\Q$ of discriminant $D=10$. Each involution $w_m$ has $2$ fixed points: $w_2$ has two fixed points which are the $-8$-CM points on this curve, $w_5$ fixes the $-20$-CM points, and $w_{10}$ fixes the $-40$-CM points. We compute \cref{tab:stacky_points}, analogous to that given in \cite[\S 3.1]{BG08} in the $D=6$ case.

\begin{table}[h]
\centering
\begin{tabular}{ccccccc}
\toprule
$W$ & $v_2$ & $v_3$ & $s_{4}$ & $s_{6}$ & $s_{12}$ & $s_{18}$ \\
\midrule
$\{\id\}$ & $0$ & $4$ & $1$ & $3$ & $5$ & $7$\\
$\langle w_2\rangle$ & $2$ & $2$ & $1$ & $1$ & $3$ & $3$\\
$\langle w_5\rangle$ & $2$ & $2$ & $1$ & $1$ & $3$ & $3$\\
$\langle w_{10}\rangle$ & $2$ & $2$ & $1$ & $1$ & $3$ & $3$\\
$\AL(\cX^{10})$ & $3$ & $1$ & $1$ & $0$ & $2$ & $1$\\
\bottomrule
\end{tabular} 
\caption{Number $v_n$ of $\mu_n$-stacky points and dimension $s_m$ of weight $m$ modular forms on $\cX^{10}_W$. }
\label{tab:stacky_points}
\end{table}

From \cref{tab:stacky_points} we can fix a weight $4$ modular form $g_4$ and weight $6$ modular forms $g_6, h_6$ and $k_6$ such that the Atkin--Lehner involutions act as follows:
\begin{align*}
    & g_4=w_2(g_4)=w_5(g_4)=w_{10}(g_4)\\
    & g_6=w_2(g_6)=-w_5(g_6)=-w_{10}(g_6)\\
    & h_6=-w_2(h_6)=w_5(h_6)=-w_{10}(h_6)\\
    & k_6=-w_2(k_6)=-w_5(k_6)=w_{10}(k_6).
\end{align*}

As noted in \cite[\S 4]{BG08}, we may normalize these forms so that 
\[ 
g_6^2+2h_6^2+k_6^2=0 \qquad \text{and} \qquad  4g_4^3+27h_6^2+k_6^2=0.
\]
These forms generate the ring of modular forms corresponding to $\cX^{10}$. With these relations, and the observation that in this case
\[ 
s_{k+12} = s_k+4, 
\]
we compute the canonical ring for each Atkin--Lehner quotient:
\begin{align*}
    R_{\{\id\}}^{10}&\cong \Q[g_4,g_6,h_6,k_6]/(g_6^2+2h_6^2+k_6^2, 4g_4^3+27h_6^2+k_6^2)\\
    R^{10}_{\langle w_2\rangle}
    &\cong \Q[g_4, g_6,h_6k_6]/(32g_4^6-20g_4^3g_6^2+27g_6^4+625(h_6k_6)^2)\\
    &\cong   \Q[x,y,z]/(32x^6-116x^3y^2+27y^4+625z^2)\\
    R^{10}_{\langle w_5\rangle}
    &\cong \Q[g_4, h_6,g_6k_6]/(16g_4^6+208g_4^3h_6^2+675h_6^4-(g_6k_6)^2)\\
    &\cong \Q[x, y,z]/(16x^6+208x^3y^2+675y^4-z^2)\\
    R^{10}_{\langle w_{10}\rangle}
    &\cong \Q[g_4, k_6,g_6h_6]/(32g_4^6-92g_4^3k_6^2-25k_6^4+729(g_6h_6)^2)\\
    &\cong \Q[x, y,z]/(32x^6-92x^3y^2-25y^4+729z^2)\\
     R^{10}_{W}&\cong \Q[g_4, g_6^2,g_6h_6k_6]/(352g_4^6g_6^4-172g_4^3g_6^4+21g_6^6-(g_6h_6k_6)^2)\\
    &\cong \Q[x, y,z]/(32x^6y-116x^3y^2+27y^3+625z^2).
\end{align*}

Let $j$ denote the weight zero modular form $g_4^3/g_6^2$. From \cite[Proposition 4.2, Proposition 4.6]{BG08} $j$ induces an isomorphism $j:C^{10}\to\mathbb{P}^1-\{0,\infty\}$. Set
\begin{align*}
    J_2&\defeq 12j^2-16j+12,\\
    J_4&\defeq 6j^4-16j^3+6j^2-16j+6,\\
    J_6 &\defeq 4j^6-16j^5+32j^3-8j^2-16j+4,\\
    J_{10}&\defeq j^4.
\end{align*}
This defines a morphism 
\begin{align*}
\sJ: \mathbb{P}^1-\{0,\infty\} &\to \cP(1,2,3,5)\\
j &\mapsto [J_2 : J_4 : J_6 : J_{10}]. 
\end{align*}

Using \cite[Proposition 4.4]{BG08}, we compute the Mestre obstruction on each $C^{10}_W$ in the same way as in \cref{prop:Mestre_obstructions_6}.

\begin{prop}\label{prop:Mestre_obstructions_10}
  For each subgroup $W$ of the Atkin--Lehner group $\AL(\cX^{10})$, outside of finitely many $j$, the Mestre obstruction on $C^{10}_{W}$ can be described by the Hilbert symbol given in \cref{tab:Mestre_obstructions_10}.
\end{prop}

\begin{table}[h]
\centering
\renewcommand{\arraystretch}{1.2}
\begin{tabular}{cc}
\toprule
$W$ & Mestre obstruction\\
\midrule
$\{\id\}$ & $(-10,5)_K$ \\
$\langle w_2 \rangle$ & $(-10(y^2-4x^3),5(8x^3-27y^2))_K$\\
$\langle w_5 \rangle$ & $(-10(25y^2-4x^3), 5)_K$\\
$\langle w_{10}\rangle$ & $(-10, 15(8x^3-25y^2))_K$\\
$\AL(\cX^{10})$ & $ (-10(y-4x^3),5(8x^3-27y))_K$\\
\bottomrule
\end{tabular}
\caption{Mestre obstructions on $C_W^{10}$.}
\label{tab:Mestre_obstructions_10}
\end{table}

Letting $\pi:\mathcal{B}\to\mathbb{P}^1$ denote the conic bundle corresponding to the Mestre obstruction on $C^{10}_{W}$, we obtain a commutative diagram analogous to (\ref{diagram:commutative_triangle}). 


\begin{theorem}\label{main_thm_10}
We have that 
\[
\overline{\mathcal{N}}^{10}(W,K,B)\ll
\begin{cases}
   B/\log(B) &\text{ if } W\in\{\AL(\cX^{10}),\langle w_2\rangle\}, \\
    B/\log(B)^{1/2} &\text{ if } W\in\{\langle w_5\rangle, \langle w_{10}\rangle\},\\
    B &\text{ if } W=\{\textnormal{id}\} \text{ and } (-10,5)_K=1, \\
    0 &\text{ if } W=\{\textnormal{id}\} \text{ and } (-10,5)_K=-1.
\end{cases}
\]
These bounds are sharp when $K=\Q$ or $W=\{\id\}$.
\end{theorem}

\begin{proof}
      For $W\neq \{\textnormal{id}\}$ we apply \cref{prop:conic_bundle_count} with $\varphi=\sJ$. By the definition of $\sJ$ we see that $\delta=2$. We now compute $\Delta(\pi)$. Note that the smooth fibers of $\pi$ will only have a single irreducible component, and thus $\delta_x(\pi)=1$ for these fibers. It remains to study the geometrically reducible fibers. From \cref{prop:Mestre_obstructions_10}, we have that the conic bundle $\pi:\mathcal{B}\to \mathbb{P}^1_j$ has exactly two reducible fibers when $W\in\{\AL(\cX^{10}),\langle w_2\rangle\}$ and exactly one reducible fiber when $W\in\{\langle w_5\rangle, \langle w_{10}\rangle\}$. In each of these cases the reducible fiber $\pi^{-1}(x)_{\overline{K}}$ has two irreducible components and $\Gal(\overline{K}/K)$ acts transitively on these components, so that $\delta_x(\pi)=1/2$. It follows that 
    \[
    \Delta(\pi)=\begin{cases}
        1 \text{ if } W\in \{\AL(\cX^{10}),\langle w_2\rangle\},\\
        1/2 \text{ if } W\in\{\langle w_5\rangle, \langle w_{10}\rangle\}.
    \end{cases}
    \]
    Applying \cref{prop:conic_bundle_count} gives the desired result.

      When $W=\{\id\}$ we either have $(-10,5)_K=1$, in which case the Mestre obstruction vanishes and the desired result follows from \cref{prop:conic_bundle_count}, or we have $(-10,5)_K=-1$, in which case $C^{10}_{\{\id\}}(K)=\emptyset$. 
\end{proof}

We now count $K$-isomorphism classes.

\begin{theorem}\label{main_thm_10_K}
We have that 
\[
\mathcal{N}^{10}(W,K,B)\ll
\begin{cases}
   B &\text{ if } W\in\{\AL(\cX^{10}),\langle w_2\rangle\}, \\
    B\log(B)^{1/2} &\text{ if } W\in\{\langle w_5\rangle, \langle w_{10}\rangle\},\\
    B\log(B) &\text{ if } W=\{\textnormal{id}\} \text{ and } (-10,5)_K=1, \\
    0 &\text{ if } W=\{\textnormal{id}\} \text{ and } (-10,5)_K=-1.
\end{cases}
\]
\end{theorem}

\begin{proof}
 
    From \cite[Proposition 4.11]{BG08} there are finitely many curves corresponding to points of $C^{10}_W$ with geometric automorphism group $C_2\times C_2$ or $D_4$, and outside of these exceptions the automorphism group is $C_2$. From \cref{proposition:Wright} the contribution to $\cN^{10}(W,K,B)$ from curves with $\Aut_{\overline{K}}(C)\cong C_2\times C_2$ is $o(B)$. From \cref{proposition:D4} the contribution from curves with $\Aut_{\overline{K}}(C)\cong D_4$ is also $o(B)$. From \cref{proposition:Wright}, \cref{main_thm_6}, and \cref{lem:product}, the contribution from curves with $\Aut_{\overline{K}}(C)\cong C_2$ gives the bounds in the statement of the theorem.
\end{proof}

\subsection{The discriminant $22$ case}\label{22_sec}



 Assume the same setup and notation as in the $D=10$ case. The curve $\cX^{22}$ has $2$ geometric points with CM by discriminant $-4$ and $4$ geometric points with CM by discriminant $-3$. As $2 \mid 22$, we will encounter $\mu_4$-stacky points on the quotient by $\langle w_2 \rangle$ and on the full quotient. Each involution $w_m$ has $2$ fixed points: $w_2$ fixes the two $-4$-CM points, and $w_{11}$ and $w_{22}$ fix the $-11$-CM and $-88$-CM points, respectively, on $\cX^{22}$. 

\begin{table}[h]
\centering
\begin{tabular}{cccccccccc}
\toprule
$W$ & $v_2$ & $v_3$ & $v_4$ & $s_{4}$ & $s_{6}$ & $s_8$ & $s_{10}$ & $s_{12}$ & $s_{18}$ \\
\midrule
$\{\text{id}\}$ & $2$ & $4$ & $0$ & $3$ & $5$ & $5$ & $7$ & $11$ & $15$\\
$\langle w_2\rangle$ & $0$ & $2$ & $2$ & $1$ & $3$ & $3$ & $3$ & $5$ & $7$\\
$\langle w_{11}\rangle$ & $3$ & $2$ & $0$ & $2$ & $2$ & $3$ & $3$ & $6$ & $7$ \\
$\langle w_{22}\rangle$ & $3$ & $2$ & $0$ & $2$ & $2$ & $3$ & $3$ & $6$ & $7$ \\
$\AL(\cX^{22})$ & $2$ & $1$ & $1$ & $1$ & $1$ & $2$ & $1$ & $3$ & $3$ \\
\bottomrule
\end{tabular} 
\caption{Number $v_n$ of $\mu_n$-stacky points and dimension $s_m$ of weight $m$ modular forms on $\cX^{22}_W$. }
\label{tab:stacky_points_22}
\end{table}

\cref{tab:stacky_points_22} displays counts of stacky points and dimensions of spaces of modular forms on each quotient in even weights up to $18$.

From this information, we fix weight $4$ forms $f_{4,i}$ for $i \in \{1,2,4\}$ generating the space of weight $4$ forms on $\cX^{22}$ and weight $6$ forms $f_{6,j}$ for $j \in \{1,2,3,4,5\}$ generating the space of weight $6$ forms on $\cX^{22}$ such that
\begin{align*}
    f_{4,1} &= w_2(f_{4,1}) = w_{11}(f_{4,1}) = w_{22}(f_{4,1}), \\
    f_{4,2} &= -w_2(f_{4,2}) = w_{11}(f_{4,2}) = -w_{22}(f_{4,2}), \\
    f_{4,3} &= -w_2(f_{4,3}) = -w_{11}(f_{4,3}) = w_{22}(f_{4,3}), \\
    f_{6,1} &= w_2(f_{6,1}) = -w_{11}(f_{6,1}) = -w_{22}(f_{6,1}), \\
    f_{6,2} &= w_2(f_{6,2}) = -w_{11}(f_{6,2}) = -w_{22}(f_{6,2}), \\
    f_{6,3} &= -w_2(f_{6,3}) = w_{11}(f_{6,3}) = -w_{22}(f_{6,3}) ,\\
    f_{6,4} &= w_2(f_{6,4}) = w_{11}(f_{6,4}) = w_{22}(f_{6,4}) , \text{ and }\\
    f_{6,5} &= -w_2(f_{6,5}) = -w_{11}(f_{6,5}) = w_{22}(f_{6,5}).
\end{align*}

As in the $D=10$ case, it follows from the dimension formula \cref{cusp_dimension_formula} that the full ring of modular forms corresponding to $\cX^{22}$ is generated by the weight $4$ and weight $6$ forms $g_4,h_4,k_4,g_6,h_6,j_6,k_6,$ and $\ell_6$. Moreover, in this case we have
\[ 
s_{k+12} = s_{k} + 10.
\]
The dimensions of these spaces of forms make it difficult to determine all of the relations without further information, compared to the $D=6,10$ cases. Fortunately, these are worked out in thesis work of Joan Nualart Riera \cite{NR15} using explicit computations of expansions of automorphic forms around CM points. Our notation for the generators has been chosen to match that in \cite{NR15}, and the relations are determined in \cite[Proposition 6.17]{NR15} to be as follows:
\begin{alignat*}{3}
    0 &= r_1 \defeq 11 f_{4,3}^2+f_{4,1}^2+11f_{4,2}^2  , \qquad  &&0 = r_{13} \defeq 176f_{4,1}^2f_{4,2} - 121f_{6,3}f_{6,4} - 27f_{6,2}f_{6,5}   , \\
    0 &= r_2 \defeq -f_{4,2}f_{6,2} + f_{4,1}f_{6,5}  , \qquad &&0 = r_{14} \defeq -f_{6,5}^2+f_{6,1}f_{6,2}  , \\
    0 &= r_3 \defeq -f_{4,2}f_{6,5} + f_{4,1}f_{6,1}  , \qquad &&0 = r_{15} \defeq -f_{6,4}f_{6,5} + f_{6,1}f_{6,3}  , \\
    0 &= r_4 \defeq f_{4,2}f_{6,4} - f_{4,3}f_{6,1}  , \qquad &&0 = r_{16} \defeq f_{6,2}f_{6,4} - f_{6,3}f_{6,5}  , \\
    0 &= r_5 \defeq f_{4,1}f_{6,4} - f_{4,3}f_{6,5}  , \qquad &&0 = r_{17} \defeq -11f_{6,3}^2 + 121f_{6,4}^2 - f_{6,2}^2 + 121f_{6,1}^2  , \\
    0 &= r_6 \defeq -f_{4,2}f_{6,3} + f_{4,1}f_{6,4}  , \qquad &&0 = r_{18} \defeq  11f_{6,3}f_{6,4} + f_{6,2}f_{6,5} + 11f_{6,1}f_{6,5}  , \\
    0 &= r_7 \defeq f_{4,1}f_{6,3} - f_{4,3}f_{6,2}  , \qquad &&0 = r_{19} \defeq 11f_{6,3}^2+f_{6,2}^2+11f_{6,5}^2  , \\
    0 &= r_8 \defeq 11f_{4,3}f_{6,4} + f_{4,1}f_{6,5} + 11f_{4,2}f_{6,1}  , \qquad &&0 = r_{20} \defeq 176f_{4,1}^3 - 121f_{6,3}^2 - 27f_{6,2}^2  , \\
    0 &= r_9 \defeq  11f_{4,3}f_{6,3}+f_{4,1}f_{6,2}+11f_{4,2}f_{6,5} , \\
    0 &= r_{10} \defeq 1936f_{4,1}f_{4,2}^2 + 297f_{6,3}^2 - 1331f_{6,4}^2+27f_{6,2}^2  , \\
    0 &= r_{11} \defeq 176f_{4,1}f_{4,2}f_{4,3} - 16f_{6,2}f_{6,4}+121f_{6,1}f_{6,4}  , \\ 
    0 &= r_{12} \defeq 176f_{4,1}^2f_{4,3}-16f_{6,2}f_{6,3}+121f_{6,1}f_{6,3} . 
\end{alignat*}

Let $\mathcal{R} = \{r_i \mid 1 \leq i \leq 20\}$ denote the complete set of $20$ relations among the generators. The canonical rings for each quotient are then  
 
\begin{align*} 
R_{\{\id\}}^{22} &\cong \Q[f_{4,1},f_{4,2},f_{4,3},f_{6,1},f_{6,2},f_{6,3},f_{6,4},f_{6,5}]/\mathcal{R} \\
R_{\langle w_2 \rangle}^{22} &\cong \Q[f_{4,1},f_{4,2}^2,f_{4,2}f_{4,3},f_{6,1},f_{6,2},f_{6,4}]/\mathcal{R}_{\langle w_2 \rangle} \\
R_{\langle w_{11} \rangle}^{22} &\cong \Q[f_{4,1},f_{4,2},f_{6,3},f_{6,4},f_{6,1}^2]/\mathcal{R}_{\langle w_{11} \rangle}\\
R_{\langle w_{22} \rangle}^{22} &\cong \Q[f_{4,1},f_{4,3},f_{6,4},f_{6,5}]/\mathcal{R}_{\langle w_{22} \rangle} \\
R_{W}^{22} &\cong \Q[f_1^2,f_2^2,f_{6,4},f_{4,2}f_{4,3}f_{6,1},f_{4,2}f_{4,3}f_{6,2},f_{4,2}f_{6,1}f_{6,5},f_{4,2}f_{6,2}f_{6,5},f_{4,3}f_{6,1}f_{6,3}, \\
& \qquad \qquad f_{4,3}f_{6,2}f_{6,3},f_{6,1}f_{6,3}f_{6,5},f_{6,2}f_{6,3}f_{6,5}]/\mathcal{R}_{\AL(\cX^{22})}, 
\end{align*}
where, for each non-trivial subgroup $W \leq \AL(\cX^{22})$, we write $\mathcal{R}_W$ for the ideal of relations in $R$ fixed by $W$. (While identifying explicit generators for the canonical ring of each quotient will be useful in what follows, writing down a full set of generators for each $\mathcal{R}_W$ appears to be computationally difficult and will not be necessary.)

For a wide range of indefinite quaternion discriminants $D$, Igusa invariants and the Mestre obstruction on $\cX^{D}$ are worked out explicitly in \cite{LY20} in terms of a specified Hauptmodul on the full quotient. In terms of a specific hauptmodul $j$ on $\cX^{22}/\AL(\cX^{22})$, Igusa invariants for $C^{22}$ can be derived from \cite{LY20}; they are
\begin{align*}
    J_2 &\defeq 2^2 (3 j^{5} + 5 j^{4} + 23 j^{3} + 10 j^{2} + 4 j + 3), \\
    J_4 &\defeq 2 (3j^{10} + 7j^{9} + 27j^{8} + 57j^{7} + 251j^{6} + 261j^{5} + 120j^{4} + 27j^{3} + 4j^{2} + 8j + 3)
, \\
    J_6 &\defeq  -2^{2} \cdot (j^{5} - 3j^{4} - 3j^{3} - 6j^{2} - 4j - 1)\\
    &\qquad \cdot (j^{10} + 5j^{9} - 9j^{8} - 69j^{7} - 87j^{6} - 89j^{5} - 36j^{4} + 17j^{3} + 12j^{2} - 1)
, \\
 J_{10} &\defeq -(j - 1)^{4} \cdot j^{6} \cdot (j + 1)^{12}. 
\end{align*}
The Igusa invariants induce a morphism 
\begin{align*}\label{eq:J22}
\sJ: \mathbb{P}^1-\{0,\infty\} &\to \cP(1,2,3,5)\\
j &\mapsto [J_2 : J_4 : J_6 : J_{10}]. 
\end{align*}

The Mestre obstruction on $\cX^{22}$ is given (in \cite[Table 4]{LY20}) by the symbol 
\[ 
\left( -22j, -11(16j + 11) \right).
\]

To make use of the explicit computation of the Mestre obstruction, as well as the above explicit computation of canonical rings on each quotient, we will relate this $j$ from \cite{LY20} to the Hauptmodul $t_{22}^+$ on the full quotient defined in \cite[Theorem 2.4]{NR15}. 

\begin{lemma}\label{Lemma: hauptmodul_relation}
 The Hauptmoduln $j$ and $t_{22}^+$ satisfy the relation
 \[ j = \dfrac{11}{16(t_{22}^+-1)} . \]
\end{lemma}
\begin{proof}
    Each of $j$ and $t_{22}^+$ is a Hauptmodul on $\cX^{22}/\AL(\cX^{22})$, so there are $a,b,c,d$ with 
    \[ t_{22}^+ = \dfrac{aj + b}{cj + d}. \]
    Knowing the values of $t$ and $j$ at three distinct points will then determine the relationship. Specifically, we use values at the three CM points that specify $j$ in \cite[Table 2]{LY20}. The values of $t_{22}^+$ at these same CM points are given in \cite[Theorem 3.1]{NR15}, and we copy the values of each in \cref{tab:CM_values_non-genus_2}.
\end{proof}

\begin{table}[h]
\centering
\renewcommand{\arraystretch}{1.6}
\begin{tabular}{ccc}
\toprule
CM order & $t_{22}^+$ & $j$\\
\midrule
$\Z\left[\sqrt{-1}\ \right]$ & $1$ & $\infty$ \\
$\Z\left[\frac{1+\sqrt{-3}}{2}\right]$ & $27/16$ & $1$ \\
$\Z\left[\frac{1+\sqrt{-11}}{2}\right]$ & $\infty$ & $0$ \\
\bottomrule 
\end{tabular}
\caption{Values of Hauptmoduln at some CM points of $\cX^{22}$.}
\label{tab:CM_values_non-genus_2}
\end{table}

\begin{prop}\label{proposition:22_extra_auts}
    Let $\alpha$ be such that
    \begin{equation*}
        121+1210\alpha+2377\alpha^2=0,
    \end{equation*}
    let $\beta$ be such that
    \begin{equation*}
        121+1331\beta+3664\beta^2,
    \end{equation*}
    and let $\gamma$ be such that
    \begin{equation*}
        1331+37389\gamma+296340\gamma^2 + 732304\gamma^3.
    \end{equation*}
    Then \cref{tab:CM_values_genus_2} lists all points of $C_W^{22}$ corresponding to genus two curves whose geometric automorphism group $\Aut_{\overline{K}}(C)$ is not isomorphic to the cyclic group of order two $C_2$.
\end{prop}

\begin{proof}
    The curves with extra automorphism are given in \cite[Proposition 3.16]{NR15}. The corresponding values of $j$ can be determined from \cite[Proposition 3.16]{NR15}, \cite[Theorem 3.10]{NR15}, and \cref{Lemma: hauptmodul_relation}. The geometric automorphism groups are then found by computing the corresponding Igusa invariants.
\end{proof}

\begin{table}[h]
\centering
\renewcommand{\arraystretch}{1.5}
\begin{tabular}{ccc}
\toprule
CM order & $j$ & $\Aut_{\overline{K}}(C)$ \\
\midrule
$\Z\left[3\sqrt{-1}\ \right]$ & $1/3$ & $D_4$ \\
$\Z\left[\frac{1+3\sqrt{-3}}{2}\right]$ & $4$ & $C_2\times C_2$ \\
$\Z\left[\sqrt{-22}\ \right]$ & $-11/16$ & $C_2\times C_2$ \\
$\Z\left[\sqrt{-14}\ \right]$ & $\frac{2377\alpha + 506}{99}$ & $C_2\times C_2$ \\
$\Z\left[\frac{1+5\sqrt{-3}}{2}\right]$ & $\frac{3664\beta + 627}{99}$ & $C_2\times C_2$ \\
$\Z\left[\frac{1+\sqrt{-59}}{2}\right]$ & $\frac{732304\gamma^2 + 296340\gamma + 29645}{1089}$ & $C_2\times C_2$ \\
\bottomrule 
\end{tabular}
\caption{Points of $\cC^{22}$ with extra automorphisms.}
\label{tab:CM_values_genus_2}
\end{table}

Using the definitions given in \cite[Proposition 6.17]{NR15} for the generators of the rings of modular forms in terms of $t_{22}^+$, we find that 
\[ t_{22}^+ = -\dfrac{f_{6,4}^2}{f_{6,1}^2}.\]
Therefore, from \cref{Lemma: hauptmodul_relation}, we have
\begin{equation}\label{Equation: j_to_forms}
j = -\dfrac{11f_{6,1}^2}{16(f_{6,1}^2+f_{6,4}^2)}.
\end{equation} 

We next compute the Mestre obstructions on each quotient. 

\begin{prop}\label{prop:Mestre_obstructions_22}
  For each subgroup $W$ of the Atkin--Lehner group $\AL(\cX^{22})$, outside of finitely many $j$, the Mestre obstruction on $C^{22}_{W}$ can be described by the Hilbert symbol given in \cref{tab:Mestre_obstructions_22}.
\end{prop}
\begin{proof}
    On the top curve, we compute
    \begin{align*}
        \left( -22j, -11(16j+11)\right) &= 
         \left( 2(f_{6,1}^2+f_{6,4}^2) , (f_{6,1}^2+f_{6,4}^2)(2f_{6,1}^2+f_{6,4}^2) \right) \\
        &= \left( 2\left(11f_{6,3}^2+f_{6,2}^2\right), \left(11f_{6,3}^2+f_{6,2}^2\right)\left(f_{6,1}^2 + \frac{1}{121}(11f_{6,3}^2+f_{6,2}^2\right) \right) \\
        &= \left( -22, f_{6,5}^2-11f_{6,1}^2\right) \\
        &= \left( -22 , f_{6,1}(f_{6,2} - 11 f_{6,1})\right).
    \end{align*} 

    The relations $r_{17}, r_{19}$ and $r_{14}$ used above are all fixed by the full Atkin--Lehner group. The only algebraic steps in our simplification that is not necessarily valid on each quotient is the identification of $f_{6,5}^2$ as a square in the canonical ring and factoring $f_{6,1}f_{6,2} - 11f_{6,1}^2$. Note that $f_{6,5}^2$ is a square only for $W \in \{\{\id\}, \langle w_{22} \rangle \}$. These observations give us the claimed symbols on each quotient in \cref{tab:Mestre_obstructions_22}.
\end{proof}

\begin{table}[h]
\centering
\renewcommand{\arraystretch}{1.2}
\begin{tabular}{cc}
\toprule
$W$ & Mestre obstruction\\
\midrule
$\{\id\}$ & $\left( -22 , f_{6,1}(f_{6,2} - 11 f_{6,1})\right)$ \\
$\langle w_2 \rangle$ & $\left( -22f_{6,5}^2 , f_{6,1}(f_{6,2} - 11 f_{6,1})\right)$ \\
$\langle w_{11} \rangle$ & $\left( -22 f_{6,5}^2, f_{6,1}f_{6,2} - 11 f_{6,1}^2\right)$ \\
$\langle w_{22}\rangle$ & $\left( -22 , f_{6,1}f_{6,2} - 11 f_{6,1}^2\right)$\\
$\AL(\cX^{22})$ & $\left( -22f_{6,5}^2 , f_{6,1}f_{6,2} - 11 f_{6,1}^2\right)$\\
\bottomrule 
\end{tabular}
\caption{Mestre obstructions on $C_W^{22}$.}
\label{tab:Mestre_obstructions_22}
\end{table}

In this case, we see that the Mestre obstruction on the top curve $\cX^{22}$ over a fixed number field $K$ is not constant, as it was for $D\in\{6,10\}$. The non-split fibers in this case do not all correspond to rational points on the coarse space $X^{22}$, as the curve has \emph{no} $\Q$-rational points. For this reason our counting result \cref{prop:conic_bundle_count} does not give the correct count over all number fields. The non-split fibers occurring when $D\in\{6,10\}$ correspond to rational points with CM by certain imaginary quadratic orders of class number one on quotients $X^D/W$ by \emph{non-trivial} subgroups $W \leq \AL(\cX^D)$. 

The residue field of the single quadratic point on $\cX^{22}$ with CM by the order $\Z\left[\frac{1+\sqrt{-11}}{2}\right]$ is $\Q(\sqrt{-11})$ (which is its own Hilbert class field, as this is a class number $1$ order).

\begin{theorem}\label{main_thm_22}
We have that
\[
\overline{\mathcal{N}}^{22}(W,K,B)\ll
\begin{cases}
   B^{2/5}/\log(B)^{3/2} &\text{ if } W\in\{\langle w_2\rangle,  \AL(\cX^{22})\} \\
   &\text{ or if } W=\langle w_{11}\rangle \text{ and } \Q(\sqrt{-11}) \subseteq K,\\
    B^{2/5}/\log(B) &\text{ if } W = \langle w_{22}\rangle \\
    &\text{ or if } W=\{\id\} \text{ and } \Q(\sqrt{-11}) \subseteq K.\\
    \end{cases}
\]
\end{theorem}

\begin{proof}
Following the arguments of \cref{main_thm_6} and \cref{main_thm_10} we apply \cref{prop:conic_bundle_count} with $\varphi=\sJ$. From the definition of $\sJ$ we see that $\delta=5$.  From \cref{prop:Mestre_obstructions_22}, we see that the Mestre conic bundle $\pi:\mathcal{B}\to \mathbb{P}^1_j$ has exactly two reducible fibers for all $W$. These two fibers correspond to $f_{6,2}=11f_{6,1}$ and $f_{6,1}=0$. Using the relations of the modular forms $\cR$, if $f_{6,2}=11f_{6,1}$ then $j=11/16$. Using the relations also shows that if $f_{6,1}=0$ then $f_{6,4}=0$. In order to find the corresponding value of $j$ when $f_{6,1}=0$ we compute $q$-expansions of the modular forms $f_{6,1}$ and $f_{6,2}$ at the point with CM by $\Z\left[\frac{1+\sqrt{-11}}{2}\right]$, which can be done using results in Nualart Riera's thesis \cite[Proposition 6.17 and Table 6.15]{NR15}. Carrying out this computation shows that if $f_{6,1}=0$ then $j=0$.

An assumption in \cref{prop:conic_bundle_count} is that all the non-split fibers lie above rational points. Although $j=0$ is a rational point on $\bbP^1_j$, we ultimately care about counting the corresponding abelian surfaces. The point $j=0$ corresponds to the single quadratic point $z_{11}$ with CM by $\Z\left[\frac{1+\sqrt{-11}}{2}\right]$ on $\cX^{22}$. This CM point has residue field $\Q(\sqrt{-11})$ (moreover $K$ is a field of definition for a QM abelian surface inducing $z_{11}$ if and only if $\Q(\sqrt{-11})\subseteq K$ by Jordan's result \cite[Theorem 2.1.3]{Jor}). Each nontrivial Atkin--Lehner involution aside from $w_{11}$ acts on $z_{11}$ as complex conjugation, while $w_{11}$ fixes $z_{11}$. We then have a single rational $-11$-CM point on the quotient by $W$ for $W \in \{\langle w_2 \rangle, \langle w_{22} \rangle, \AL(\cX^{22})\}$ which is elliptic of order $2$ in the case of the full quotient, and a single quadratic $-11$-CM point which is elliptic of order $2$ on the quotient by $\langle w_{11} \rangle$. 

We now compute $\Delta(\pi)$. If $W=\{\id\}$ and $\Q(\sqrt{-11}) \subseteq K$, or if $W=\langle w_{22}\rangle$, then the two reducible fibers each have two reducible components on which  $\Gal(\overline{K}/K)$ acts transitively. Therefore $\Delta(\pi)=1/2+1/2=1$.

If $W\in\{\langle w_2\rangle, \langle w_{11}\rangle, \AL(\cX^{22})\}$ then one of the reducible fibers has a single component of multiplicity 2, and the other reducible fiber has two components on which  $\Gal(\overline{K}/K)$ acts transitively. Therefore $\Delta(\pi)=1+1/2=3/2$.
\end{proof}

We now count $K$-isomorphism classes.

\begin{theorem}\label{main_thm_22_K}
Assume that
\begin{itemize}
\item $W\in \{\langle w_2\rangle, \langle w_{22}\rangle, \AL(\cX^{22})\}$ or
\item $W\in \left\{ \{\id\}, \langle w_{11}\rangle\right\}$ and $\Q(\sqrt{-11})\subseteq K$.
\end{itemize}
Then
\[
\mathcal{N}^{22}(W,K,B)\asymp B.
\]
\end{theorem}

\begin{proof}
 
    From  \cref{proposition:22_extra_auts} there are finitely many curves corresponding to points of $C^{22}_W$ with geometric automorphism group $C_2\times C_2$ or $D_4$, and outside of these exceptions the automorphism group is $C_2$. From \cref{proposition:Wright} the contribution to $\cN^{22}(W,K,B)$ from curves with $\Aut_{\overline{K}}(C)\cong C_2\times C_2$ is $o(B)$. From \cref{proposition:D4} the contribution from curves with $\Aut_{\overline{K}}(C)\cong D_4$ is also $o(B)$. From \cref{proposition:Wright}, \cref{main_thm_22}, and \cref{lem:product}, the contribution from curves with $\Aut_{\overline{K}}(C)\cong C_2$ gives the estimate in the statement of the theorem.
\end{proof}

\bibliographystyle{amsalpha}
\bibliography{biblio}
\end{document}